\documentclass[10pt]{amsart}
\usepackage{enumerate,amsmath,amssymb,latexsym,
amsfonts, amsthm, amscd, MnSymbol}

\theoremstyle{plain}

\newtheorem{theorem}{Theorem}

\numberwithin{equation}{section}

\newcommand{\ra}{\rightarrow}

\addtocounter{section}{-1}

\begin{document}

\title {Remarks on trigonometric functions after Eisenstein}

\date{}

\author[P.L. Robinson]{P.L. Robinson}

\address{Department of Mathematics \\ University of Florida \\ Gainesville FL 32611  USA }

\email[]{paulr@ufl.edu}

\subjclass{} \keywords{}

\begin{abstract}

We modify the Whittaker-Watson account of the Eisenstein approach to the trigonometric functions, basing these functions independently on the Eisenstein function $\varepsilon_2$. 

\end{abstract}

\maketitle

\medbreak

\section{Introduction} 

Eisenstein [E] initiated a novel approach to the theory of the trigonometric functions, based on the meromorphic functions defined by 
$$\varepsilon_k (z) = \sum_{n \in \mathbb{Z}} \frac{1}{(z + n)^k}$$
for $k$ a positive integer and $z \in \mathbb{C} \setminus \mathbb{Z}$. These functions were named in honour of Eisenstein by Weil, who elaborated details of the somewhat mystical calculations and further developed the theory in [W]. Of course, this novel approach to the trigonometric functions was but an offshoot or a shadow of the larger theory of elliptic functions. In their account of the Weierstrassian elliptic function theory, Whittaker and Watson [WW] include a very brief introduction to this trigonometric theory by way of illustration. 

\medbreak 

A little more explicitly, the approach of [E] as explicated in [W] develops the theory of trigonometric functions from the fundamental formula 
$$\varepsilon_1 (z) = \pi \cot \pi z.$$
This formula is intended as a {\it definition} of the cotangent function in terms of the positive constant $\pi$ {\it defined} by 
$$\pi^2 = 6 \sum_{n = 1}^{\infty}\frac{1}{n^2}.$$
For the identification of $\varepsilon_1 (z)$ with $\pi \cot \pi z$ as it is ordinarily understood, we refer to Remmert [R]; this reference also contains an outline of the Eisenstein approach and places it in historical context. 

\medbreak 

Our purpose here is to modify the approach adopted in [WW] so as to develop the trigonometric functions from the Eisenstein series $\varepsilon_2$. The approach in [WW] does not lend itself directly to a wholly independent construction of the trigonometric functions, as it incorporates $\pi$ with its ordinary meaning and makes use of the classical formulae 
$$\sum_{n = 1}^{\infty} \frac{1}{n^2} = \frac{\pi^2}{6} \; \; {\rm and} \; \; \sum_{n = 1}^{\infty} \frac{1}{n^4} = \frac{\pi^4}{90}.$$
When the approach in [WW] is reformulated so as not to assume $\pi$ with its ordinary meaning, the proof given there requires independent knowledge of the identity 
$$2 \: \big[ \sum_{n = 1}^{\infty} \frac{1}{n^2} \big]^2 = 5 \sum_{n = 1}^{\infty} \frac{1}{n^4}.$$
Our modification circumvents the need for this independent knowledge and indeed has this identity as a consequence. The approach in [WW] essentially identifies $\varepsilon_2 (z)$ as $\pi^2 {\rm cosec}^2 \pi z$ by virtue of its satisfying certain nonlinear differential equations of first and second order. Our modification goes beyond this: the reciprocal of $\varepsilon_2$ satisfies the second-order linear differential equation
$$g'' + \big(24 \sum_{n = 1}^{\infty}\frac{1}{n^2}\big) \: g = 2$$
from which the elementary trigonometric functions are immediately in evidence. Our approach has other benefits: for example, it eliminates the need for such tools as the Herglotz trick and the maximum modulus principle, which feature in some accounts of the theory. 

\medbreak 

\section{A modified approach} 

\medbreak 

Our starting point is the second Eisenstein series, which we rename $f$ for simplicity:  
$$f(z) =  \sum_{n \in \mathbb{Z}} \frac{1}{(z - n)^2}$$
for $z \in \mathbb{C} \setminus \mathbb{Z}$. The indicated series is normally convergent: let $K \subseteq \mathbb{C} \setminus \mathbb{Z}$ be compact and choose $R > 0$ so that $K$ lies in the disc $D_R (0)$; if $z \in K$ and $|n| > R$ then $|z - n| > |n| - R$ so that
$$\sum_{|n| > R}\frac{1}{|z - n|^2} \leqslant \sum_{|n| > R} \frac{1}{(|n| - R)^2}$$
and the uniformly majorizing series on the right converges by the limit comparison test. As a consequence, $f : \mathbb{C} \setminus \mathbb{Z} \ra \mathbb{C}$ is holomorphic; moreover, $f$ is plainly even and of period one. At each integer, $f$ has a double pole: around zero,  
$$f(z) = z^{-2} + \sum_{0 \neq n \in \mathbb{Z}} (z - n)^{-2}$$
where the second summand on the right is holomorphic in the open unit disc, there having Taylor expansion 
$$\sum_{0 \neq n \in \mathbb{Z}} (z - n)^{-2} = \sum_{d = 0}^{\infty} a_d z^{2 d}$$
with 
$$a_d = 2 (2 d + 1) \sum_{n = 1}^{\infty} n^{- (2 d + 2)}$$
as follows from the derived geometric series. 

\medbreak 

We now employ a familiar device, combining suitable derivatives and powers of $f$ so as to eliminate the poles. The Laurent expansion of $f(z)$ about the origin reads 
$$f(z) = z^{- 2} + a_0 + a_1 z^2 + \dots $$
so that 
$$f'(z) = - 2 z^{-3} + 2 a_1 z + \dots $$
and 
$$f''(z) = 6 z^{-4} + 2 a_1 + \dots $$
while 
$$f(z)^2 = z^{-4} + 2 a_0 z^{-2} + (a_0^2 + 2 a_1) + \dots .$$
\medbreak
\noindent
The combination $f'' - 6 f^2 + 12 a_0 f$ is of course holomorphic in $\mathbb{C} \setminus \mathbb{Z}$ and has period one; its singularities at the integers are removable, in view of the expansion 
$$f''(z) - 6 f(z)^2 + 12 a_0 f(z) = (6 a_0^2 - 10 a_1) + \dots $$
about the origin, where the ellipsis indicates a power series involving terms of degree two or greater. Removing these singularities, $f'' - 6 f^2 + 12 a_0 f$ becomes an entire function. 

\medbreak 

To proceed further, we examine the behaviour of $f$ in the vertical strip 
$$S = \{ z \in \mathbb{C} : |{\rm Re} \; z | \leqslant 1 \}.$$

\medbreak 

\begin{theorem} \label{zeroinfinity}
$f(z) \ra 0$ as $z \ra \infty$ in the strip $S$. 
\end{theorem} 

\begin{proof} 
Let $z = x + i y \in S$ so that $|x| \leqslant 1$ and if $n \in \mathbb{Z}$ then $|z - n|^2 = (n - x)^2 + y^2$. If $|n| \leqslant 1$ then $|z - n|^2 \geqslant y^2$ while if $|n| > 1$ then $|z - n|^2 \geqslant (|n| - 1)^2 + y^2$. Accordingly, it follows that 
$$|f(z)| \leqslant \frac{3}{y^2} + 2 \sum_{n = 1}^{\infty} \frac{1}{n^2 + y^2}.$$
As $z \ra \infty$ in $S$ we need only inspect the second summand on the right. For any $N$ we have 
$$\sum_{n = 1}^{\infty} (n^2 + y^2)^{-1} = \sum_{1 \leqslant n  \leqslant N} (n^2 + y^2)^{-1} + \sum_{n > N} (n^2 + y^2)^{-1}.$$
Let $\varepsilon > 0$: choose $N$ so that $\sum_{n > N} n^{-2} < \varepsilon$; it follows that if $|y| > \sqrt{N/\varepsilon}$ then 
$$\sum_{n = 1}^{\infty} (n^2 + y^2)^{-1} \leqslant N y^{-2} + \sum_{n > N} n^{-2} < 2 \varepsilon.$$
\end{proof} 

\medbreak 

We now see that the entire function $f'' - 6 f^2 + 12 a_0 f$ is as trivial as can be. 
\medbreak 

\begin{theorem} \label{2DE}
The meromorphic function $f$ satisfies
$$f'' - 6 f^2 + 12 a_0 f = 0.$$
\end{theorem} 

\begin{proof} 
The argument of Theorem \ref{zeroinfinity} adapts easily to show that the second derivative $f''(z) = 6 \sum_{n \in \mathbb{Z}} (z - n)^{-4}$ also tends to $0$ as $z \ra \infty$ in $S$. The entire function $f'' - 6 f^2 + 12 a_0 f$ is thus bounded in $S$ and so bounded on $\mathbb{C}$ by periodicity. According to the Liouville theorem, $f'' - 6 f^2 + 12 a_0 f$ is constant; the value of this constant is $0$ because $f'' - 6 f^2 + 12 a_0 f$ vanishes at infinity. 
\end{proof} 

\medbreak 

Thus the constant term $6 a_0^2 - 10 a_1$ in the expansion of $f'' - 6 f^2 + 12 a_0 f$ about the origin is zero. When we substitute the expressions for $a_0$ and $a_1$ and then simplify, we obtain the identity
$$2 \: \big[ \sum_{n = 1}^{\infty} \frac{1}{n^2} \big]^2 = 5 \sum_{n = 1}^{\infty} \frac{1}{n^4}.$$

\medbreak 

\begin{theorem} \label{1DE}
The meromorphic function $f$ satisfies
$$(f')^2 - 4 f^3 + 12 a_0 f^2 = 0.$$
\end{theorem} 

\begin{proof} 
Multiply the equation of Theorem \ref{2DE} by $2 f'$ to obtain 
$$2 f' f'' - 12 f' f^2 + 24 a_0 f' f = 0$$
and then integrate to obtain 
$$(f')^2 - 4 f^3 + 12 a_0 f^2 = c$$
for some $c \in \mathbb{C}$. As $f$ and (similarly) $f'$ vanish at infinity, $c = 0$. 
\end{proof} 

\begin{theorem} \label{zero} 
The function $f$ is nowhere zero. 
\end{theorem} 

\begin{proof} 
Theorem \ref{2DE} and Theorem 3 tell us that  if we write $2 p(w) = 4 w^3 - 12 a_0 w^2$ then $f'' = p' \circ f$ and $(f')^2 = 2 p \; \circ f$. An elementary induction shows that each even-order derivative of $f$ is a polynomial in $f$ with vanishing constant term: for the inductive step, if $f^{(2 d)} = q \circ f$ then $f^{(2 d + 1)} = (q' \circ f) f'$ and $f^{(2 d + 2)} = (2 q'' p + q' p') \circ f$; the square of each odd-order derivative of $f$ is then also a polynomial in $f$ with vanishing constant term. Finally, if $f$ were to vanish at $a \in \mathbb{C} \setminus \mathbb{Z}$ then all its derivatives would vanish at $a$; the Identity Theorem would then force $f$ itself to vanish, which is absurd. 
\end{proof} 

\medbreak 

We may now introduce the reciprocal function $g = 1/f$: as $f$ is a nowhere-zero meromorphic function with a double pole at each integer, $g$ is an entire function with a double zero at each integer; as $f$ is even and of period one, $g$ is even and of period one. 

\medbreak 

\begin{theorem} \label{trig}
The entire function $g$ satisfies 
$$g'' + 12 a_0 g = 2.$$
\end{theorem} 

\begin{proof} 
Simply differentiate and then substitute from Theorem \ref{2DE} and Theorem \ref{1DE}: $g' = - f^{-2} f'$ so that $g'' = 2 f^{-3} (f')^2 - f^{-2} f'' =  2 f^{-3} (4 f^3 - 12 a_0 f^2) - f^{-2} (6 f^2 - 12 a_0 f) = 2 - 12 a_0 g\;$  as required. 
\end{proof} 

\medbreak 

Recall that $g$ has a double zero at the origin; accordingly, the second-order differential equation displayed in Theorem \ref{trig} is supplemented by the initial data $g(0) = 0$ and $g'(0) = 0$. 

\medbreak 

At this point, it is quite clear that our approach has made contact with the elementary trigonometric functions. {\it Define} the positive number $\pi$ by 
$$\pi^2 := 3 a_0 = 6 \sum_{n = 1}^{\infty} n^{-2}.$$ 
{\it Define} the function $c: \mathbb{C} \ra \mathbb{C}$ by the rule that if $z \in \mathbb{C}$ then 
$$c(z) : = 1 - 2 \pi^2 g(z/2 \pi).$$
The entire function $c$ has period $2 \pi$; this inbuilt periodicity is a special feature of the Eisenstein approach. Further, a direct calculation reveals that it satisfies the initial value problem 
$$c'' + c = 0; \; \; c(0) = 1, \; c'(0) = 0.$$
As an entire function, its Taylor series about the origin is consequently 
$$c(z) = \sum_{n = 0}^{\infty} (-)^n \frac{z^{2 n}}{(2 n)!}.$$
Thus $c$ is precisely the cosine function, from which flows the whole theory of trigonometric functions. Incidentally, a duplication formula for the cosine function shows that $f(z) = \pi^2 {\rm cosec}^2 \pi z$. 

\medbreak 

We close by remarking on ways in which our approach varies from the approach in [WW]. First of all, [WW] incorporates $\pi$ in the theory from the very start; its removal from the function there analyzed yields $f$. Our Theorem \ref{zeroinfinity} improves the [WW] observation that $f(z)$ is {\it bounded} as $z \ra \infty$ in the strip $\{ z \in \mathbb{C} : | {\rm Re} \; z | \leqslant 1/2 \}$; the weaker result means that [WW] must assume the identity $2[\sum_{n = 1}^{\infty} n^{-2} ]^2 = 5 \sum_{n = 1}^{\infty} n^{-4}$ in order to conclude that the bounded function $f'' - 6 f^2 + 12 a_0 f$ is identically zero. Our Theorem \ref{zero} to the effect that $f$ never vanishes permits us to pass directly to its reciprocal $g$ and thence to the elementary second-order linear differential equation in Theorem \ref{trig}; by contrast, [WW] essentially stops short at the nonlinear differential equations that we display in Theorem \ref{2DE} and Theorem \ref{1DE}. 

\medbreak

\bigbreak

\begin{center} 
{\small R}{\footnotesize EFERENCES}
\end{center} 
\medbreak 

[E] F. G. M. Eisenstein, {\it Genaue Untersuchung der unendlichen Doppelproducte ... }, Jour. f\"ur Reine und Angew. Math. {\bf 35} (1847) 153-274. 

\medbreak 

[R] R. Remmert, {\it Theory of Complex Functions}, Graduate Texts in Mathematics {\bf 122}, Springer-Verlag (1991). 

\medbreak 

[W] A. Weil, {\it Elliptic Functions according to Eisenstein and Kronecker}, Ergebnisse der Mathematik {\bf 88}, Springer-Verlag (1976). 

\medbreak 

[WW] E. T. Whittaker and G. N. Watson, {\it A Course of Modern Analysis}, Second Edition, Cambridge University Press (1915).

\medbreak

\end{document}